\documentclass{amsart}
\usepackage{amsmath}
\newtheorem{theorem}{Theorem}[section]
\newtheorem{lemma}[theorem]{Lemma}
\newtheorem{cor}[theorem]{Corollary}
\theoremstyle{definition}

\newcommand{\p}{{\mbox{$[p]$}}}
\newcommand{\scp}{{\mbox{$\scriptstyle [p]$}}}

\DeclareMathOperator{\asoc}{Asoc}
\DeclareMathOperator{\U}{U}
\DeclareMathOperator{\redu}{u}
\DeclareMathOperator{\cl}{Cl}
\DeclareMathOperator{\id}{id}

\title{Faithful irreducible representations of modular Lie algebras}
\author{Donald W. Barnes}
\address{1 Little Wonga Rd.\\Cremorne NSW 2090\\Australia\\}
\email{D.Barnes@maths.usyd.edu.au}
\thanks{This work was done while the author was an Honorary Associate of the School of Mathematics and Statistics, University of Sydney}


\begin{document}

\begin{abstract} Let $L$ be a finite-dimensional Lie algebra over a field of characteristic $p\ne 0$.  By a theorem of Jacobson, $L$ has a finite-dimensional faithful completely reducible module.  We show that if $F$ is not algebraically closed, then $L$ has an irreducible such module.  We also give a necessary and sufficient condition for a finite-dimensional Lie algebra over an algebraically closed field of non-zero characteristic to have a faithful irreducible module.
\end{abstract}

\maketitle

\section{Introduction}

Let $L$ be a finite-dimensional Lie algebra over the field $F$ of characteristic $p>0$. By Jacobson's Theorem, \cite[Theorem VI.2]{Jac}, \cite[Theorem 5.5.2]{SF}, $L$ has a finite-dimensional faithful completely reducible module.  We adapt Jacobson's proof to show that if $F$ is not algebraically closed, then $L$ has an irreducible such module.  We may replace $L$ with a minimal $p$-envelope.  Thus without loss of generality, we may assume that $L$ is restricted.  Further,  by \cite[Corollary 2.3]{A-I dim}, we may assume that the $p$-operation $\p$ vanishes on $\asoc(L)$, the abelian socle of $L$.  We note that, by the Artin-Schreier Theorem\footnote{I wish to thank J. M. Bois for drawing my attention to this theorem.}, (see \cite[Theorem 11.14]{JacB} or \cite[Theorem 3.1]{con},) our assumption that $F$ is not algebraically closed implies that the algebraic closure $\bar{F}$ has infinite dimension over $F$.

Jacobson takes an appropriately selected character $c$ and uses the $c$-reduced enveloping algebra $\redu(L,c)$ to obtain for a given element $x \in L$, an irreducible module on which $x$ acts non-trivially.  If $x$ is in some minimal ideal $K$ of $L$, then $K$ is not in the kernel of the representation.  For each minimal ideal $K$, we can clearly construct an irreducible module $V_K$ on which $K$ acts non-trivially.  We shall choose modules $V_K$ such that their tensor product has a composition factor on which all minimal ideals act non-trivially.  It is in the making of this choice for the abelian minimal ideals that we use that the dimension of $\bar{F}$ over $F$ is infinite.

\section{diagonals}\label{diags}
Let $A_1, \dots, A_r$ be abelian minimal ideals of $L$ which, as $L$-modules,  are isomorphic with isomorphisms $\phi_i: A_1 \to A_i$, $\phi_1= \id$.  For a given $A_1$, we choose the $A_i$ with $r$ as large as possible subject to their sum being direct.  With the $A_i$ so chosen, we say that $A_1$ has multiplicity $r$ in $\asoc(L)$. For any $\lambda_1, \dots, \lambda_r $, not all zero, the set $\{\sum_{i=1}^r \lambda_i \phi_i(a) \mid a\in A_1\}$ is also a minimal ideal and every minimal ideal isomorphic to $A_1$ has this form.  Those with more that one of the $\lambda_i$ non-zero are called \textit{diagonals.}  We shall denote the ideal given by $\lambda = (\lambda_1, \dots, \lambda_r)$ by $A_\lambda$.

Under our assumptions, an abelian minimal ideal $A$ acts non-trivially on an irreducible module $V$ with character $c$ if and only if $c(A) \ne 0$.  If for the above $A_i$, we put $c_i = c| A_i$, then for $a_\lambda = \sum \lambda_i \phi_i(a) \in  A_\lambda$, we have $c(a_\lambda)= \sum \lambda_i c_i(\phi_i(a))$.  Thus $c$ vanishes on some $A_\lambda$ if and only if the $c_i\phi_i$ are linearly dependent.

We shall use generalised characters, linear maps $c:L \to \bar{F}$, where $\bar{F}$ denotes the algebraic closure of $F$.  If $c$ is a character of a module $V$ then every conjugate of $c$ is also a character.  The set of all characters of composition factors of $\bar{V}=\bar{F} \otimes V$ is called the \textit{character cluster} $\cl(V)$.  (See \cite{cluster}.)  If the restriction $\cl(V)|A$ of $\cl(V)$ to the abelian minimal ideal $A$ does not contain $0$, then $A$ acts non-trivially on every composition factor of $V$.  By \cite[Theorem 3.4]{cluster}, if $V$ is irreducible, then $\cl(V)$ is simple, that is, the characters in $\cl(V)$ are all conjugate.   Thus if $c$ is a generalised character of the irreducible module $V$ and $c(A) \ne 0$, then $A$ acts non-trivially on $V$.

If $C$ is a simple cluster, then we can form the $C$-reduced enveloping algebra $\redu(L,C)$.  For any $x \in L$, the element $x^p - x^\scp$ of the universal enveloping algebra is central and on an $\bar{F} \otimes L$-module with character $c$, acts as multiplication by $c(x)^p$.  Let $m_x(t)\in F[t]$ be the minimal polynomial of $c(x)^p$.  The $C$-reduced enveloping algebra $\redu(L,C)$ is the quotient of $\U(L)$ by the ideal generated by $\{m_x(x^p-x^\scp) \mid x \in L\}$ for some $c \in C$.   (Since all $c \in C$ are conjugate, they all give the same polynomial $m_x(t)$.)   The $C$-reduced enveloping algebra is finite-dimensional.  Any $\redu(L,C)$-module is an $L$-module with cluster $C$.  This construction may be found in Farnsteiner \cite{Farn} or in Barnes \cite{induced}.

\begin{lemma} \label{exC} For any simple cluster $C$, there exists a finite-dimensional irreducible $L$-module with cluster $C$.
\end{lemma}

\begin{proof} As $\redu(L,C)$ is a $\redu(L,C)$-module, any $L$-module composition factor of $\redu(L,C)$ is an irreducible module with cluster $C$.
\end{proof}

\begin{lemma} \label{abelian} There exists an irreducible $L$-module $V_0$ on which every abelian minimal ideal acts non-trivially. 
\end{lemma}

\begin{proof}  Let $A_1, \dots, A_r$ be abelian minimal ideals of $L$ which are isomorphic with isomorphisms $\phi_i: A_1 \to A_i$, $\phi_1 = \id$, chosen as above.  As $\dim(\bar{F}/F)$ is infinite, we can take $r$ linearly independent linear maps $f_i: A_1 \to  \bar{F}$ and by setting $g_i = f_i \circ\phi_i^{-1}: A_i \to \bar{F}$, construct a map $g:\sum_iA_i \to \bar{F}$ which is non-zero on every minimal ideal isomorphic to $A_1$.  Doing this for each isomorphism type of abelian minimal ideal, we obtain $c: \asoc(L) \to \bar{F}$ which is nonzero on every abelian minimal ideal.  This map can be extended to a generalised character and so we obtain an irreducible module $V_0$ on which every abelian minimal ideal acts non-trivially.
\end{proof}

\section{The main result}
\begin{lemma} \label{prod} Let $K$ be an ideal of $L$ and let $V,W$ be irreducible $L$-modules.  Suppose that $K$ acts non-trivially on $V$ and trivially on $W$.  Then $K$ acts non-trivially on every composition factor of $V\otimes W$.
\end{lemma}

\begin{proof} Regarded as $K$-module, $V \otimes W$ is a direct sum of copies of $V$.  Thus any composition factor of $V \otimes W$ is as $K$-module, a direct sum of copies of $V$.
\end{proof}

\begin{theorem} \label{main} Let $L$ be a finite-dimensional Lie algebra over a field $F$ of characteristic $p \ne 0$. Suppose that $F$ is not algebraically closed.  Then there exists a finite-dimensional faithful irreducible $L$-module.
\end{theorem}

\begin{proof}  Let $K_1, \dots, K_s$ be the non-abelian minimal ideals of $L$.  Then $K_i$ is an irreducible $L$-module on which $K_i$ acts non-trivially and on which every other minimal ideal acts trivially.  We start with $V_0$ given by Lemma \ref{abelian} and construct inductively irreducible modules $V_i$ on which the abelian minimal ideals and the $K_j$ for $j \le i$ all act non-trivially.

We suppose that we have constructed $V_{i-1}$.  If $K_i$ acts non-trivially on $V_{i-1}$, we set $V_i = V_{i-1}$.  If not, then we take for $V_i$ any composition factor of $V_{i-1} \otimes K_i$.  By Lemma \ref{prod}, this satisfies the requirements.

Put $V = V_s$.  Then $V$ is a finite-dimensional irreducible $L$-module on which every minimal ideal of $L$ acts non-trivially.  Since the kernel of the representation cannot contain any minimal ideal, the representation is faithful.
\end{proof}

When the field $F$ is algebraically closed field, the above argument leads to a necessary and sufficient condition for the existence of a faithful irreducible module. 

\begin{theorem} \label{aclosed} Let $L$ be a finite-dimensional Lie algebra over the algebraically closed field $F$ of characteristic $p \ne  0$.  Then $L$ has a faithful irreducible module if and only if, for each abelian minimal ideal $A$ of $L$, the $L$-module isomorphism type of $A$ has multiplicity $r \le \dim(A)$.
\end{theorem}

\begin{proof} 
We show first that if $L$ satisfies the condition, then so does a minimal $p$-envelope $(L^e,\p)$ of $L$.  We may suppose that $\p$ vanishes on $\asoc(L)$.  Suppose that $A$ is an abelian minimal ideal of $L^e$.  Then $A$ is a $\p$-ideal and $L^e/A$ is again a restricted Lie algebra.  If $A \not\subseteq L$, then $L^e/A$ is a $p$-envelope of $L$ of smaller dimension.  Thus every minimal ideal of $L^e$ is contained in $L$ and again, without loss of generality, we may suppose that $L$ is restricted with $\p$ vanishing on $\asoc(L)$.

We have already shown that the condition is sufficient for the existence of a faithful irreducible $L$-module.  So suppose that $V$ is a faithful irreducible $L$-module.  By \cite[Theorem 5.2.5]{SF}, $V$ has a character $c$.  Let $A_1, \dots , A_r$ be abelian minimal ideals chosen as above, with $L$-module isomorphisms $\phi_i: A_1 \to A_i$.  We have the linear maps $c_i = c \circ \phi_i: A_1 \to F$.  For the diagonal $A_\lambda$, $c|A_\lambda = \sum \lambda_i c_i \circ \phi^{-1}$.  As $V$ is faithful, $c$ is non-zero on every $A_\lambda$, so the $c_i$ are linearly independent.  Therefore $r \le \dim(A_1)$.
\end{proof}

Examples of algebras not satisfying the condition are easily found.  The $2$-dimensional abelian algebra is one such.

\begin{cor} Let $L$ be a finite-dimensional Lie algebra over an algebraically closed field of characteristic $p \ne 0$.  Then $L$ is a quotient of an algebra which has a faithful irreducible module.
\end{cor}

\begin{proof} $L$ has a faithful completely reducible module $V$.  We may suppose that $V$ has no redundant summands.  The split extension of $V$ by $L$ satisfies the condition for the existence of a faithful irreducible module.
\end{proof}

\bibliographystyle{amsplain}

\end{document}